\documentclass[12pt]{amsart}
\usepackage{amsmath}
\usepackage{amsthm,amssymb}
\usepackage{amssymb}
\usepackage{graphicx}
\usepackage[unicode]{hyperref}
\usepackage{textcase}
\input xy
\xyoption{all}
\usepackage{color}
\usepackage{enumerate}
\newtheorem{theorem}{Theorem}[section]

\newtheorem{lemma}[theorem]{Lemma}

\newtheorem{corollary}[theorem]{Corollary}
\newtheorem{proposition}[theorem]{Proposition}
\theoremstyle{definition}

\newtheorem{remark}[theorem]{Remark}
\newtheorem{problem}[theorem]{Problem}

\usepackage[normalem]{ulem}

\begin{document}
\title[Products of traceless and semi-traceless matrices ...]{Products of traceless and semi-traceless \\ matrices over division rings \\ and their applications}
	
\author[P. V. Danchev]{Peter V. Danchev $^{1,*}$}
\author[T. H. Dung]{Truong Huu Dung $^{2}$}
\author[T. N. Son]{Tran Nam Son $^{3,4}$}

\address{[1] Institute of Mathematics and Informatics, Bulgarian Academy of Sciences, 1113 Sofia, Bulgaria}
\address{[2] Department of Mathematics, Dong Nai University, 9 Le Quy Don Str., Tan Hiep Ward, Bien Hoa City, Dong Nai Province, Vietnam}
\address{[3] Faculty of Mathematics and Computer Science, University of Science, Ho Chi Minh City, Vietnam}
\address{[4] Vietnam National University, Ho Chi Minh City, Vietnam}

\email{\newline Peter V. Danchev: danchev@math.bas.bg or pvdanchev@yahoo.com; \newline
	Truong Huu Dung: thdung@dnpu.edu.vn or dungth0406@gmail.com; \newline
    Tran Nam Son: trannamson1999@gmail.com}
	
\keywords{Traceless matrices, Semi-traceless matrices, Fields, Division rings, Images, Non-commutative polynomials, Ver\v{s}ik-Kerov group.\\
		\protect \indent 2020 {\it Mathematics Subject Classification.} 16U60; 16S34; 16U99.\\ \protect \indent * Corresponding author: Peter V. Danchev}
	
\begin{abstract} We study the problem when every matrix over a division ring is representable as either the product of traceless matrices or the product of semi-traceless matrices, and also give some applications of such decompositions. Specifically, we establish the curious facts that every matrix over a division ring is a product of at most twelve traceless matrices as well as a product of at most four semi-traceless matrices. We also examine finitary matrices and certain images of non-commutative polynomials by applying the obtained so far results showing that the elements of some finite-dimensional algebras over a special field as well as that these of the matrix algebra over any division ring possess some rather interesting and non-trivial decompositions into products of at most four generalized commutators.
\end{abstract}

\maketitle
	
\section{Introduction and Fundamentals}
	
Throughout this paper, let $D$ be a division ring and $n>1$ a natural number. Standardly, we use the notations of $\mathrm{M}_n(D)$ and $\mathrm{GL}_n(D)$ as the ring of matrices of size $n$ over $D$ and the general linear group of matrices of size $n$ over $D$, respectively. As usual, a matrix is called \textit{traceless} (resp., \textit{semi-traceless}) if it is (resp., similar to) a matrix whose trace is zero. According to \cite[Proposition 9]{Pa_Me_13}, if $A\in\mathrm{M}_n(D)$ is non-central and traceless, then $A$ is similar to a matrix in $\mathrm{M}_n(D)$ with only $0$ on the main diagonal. However, the converse is definitely {\it not} true. For example, let $\mathbb{H}$ be the real quaternion division ring with $i,j,k$ satisfying $i^2=j^2=k^2=-1$ and $ij=-ji=k$. Then, one sees that
$$\begin{pmatrix}
		j&0\\i&1
	\end{pmatrix}^{-1}\begin{pmatrix}
		i&j\\-j&i
	\end{pmatrix}\begin{pmatrix}
		j&0\\i&1
	\end{pmatrix}=\begin{pmatrix}
		0&1\\0&0
	\end{pmatrix},$$ but the matrix $\begin{pmatrix}
		i&j\\-j&i
	\end{pmatrix}$ is not traceless. Note that similar matrices do not have same trace in $\mathrm{M}_n(D)$. For instance, if $a$ and $b$ are in $D$ that do not commute, then the matrix $\left(\begin{matrix}
		a&0\\0&-a
	\end{matrix}\right)$ is traceless, but however the product
$$\left(\begin{matrix}
	1&0\\0&b
	\end{matrix}\right)\left(\begin{matrix}
	a&0\\0&-a
	\end{matrix}\right)\left(\begin{matrix}
	1&0\\0&b
	\end{matrix}\right)^{-1}$$ is not traceless. Therefore, we shall proceed with semi-traceless matrices in $\mathrm{M}_n(D)$. So, we denote by $\mathrm{sl}_n(D)$ the set of all traceless matrices in $\mathrm{M}_n(D)$.

\medskip

The objective which motivates writing of this article is to initiate an in-depth study of the decomposing properties of matrices over division rings into products of traceless and semi-traceless matrices. We also focussing on the exploration of decomposable properties of finitary matrices over division rings and connect them with traceless and semi-traceless matrices. We apply what we achieved to the images of non-commutative polynomials. 
	
Concretely, our work is organized thus: In the next second section, we investigate the decomposition of matrices from the general linear group over an arbitrary division ring into products of traceless matrices of the matrix ring (see Theorem~\ref{trace} and Proposition~\ref{twoprod}). In the subsequent third section, we examine the decomposition of matrices over a non-commutative division ring into products of semi-traceless matrices (see Theorem~\ref{semi-trace} and Proposition~\ref{semi-2}). In the fourth section, we focus on the so-called {\it finitary} matrices over division rings by proving some deep results pertaining to their decomposition into semi-traceless matrices (see Theorems~\ref{new1}, \ref{image1} and \ref{image2}). In the fifth section, we concentrate on the images of non-commutative polynomials by applying the results from the foregoing sections in order to establish certain decompositions of elements of some algebras (see Theorem~\ref{main3}, Proposition~\ref{algebraically} and Corollary~\ref{algebraic}). We finish our study with an important problem which seems to be extremely difficult (see Problem~\ref{unsolved}).
	
\section{Products of traceless matrices}

We begin here with the following statement.
	
\begin{theorem}\label{field}
If $D$ is a field, then  then each matrix in $\mathrm{M}_n(D)$ is a product of two traceless matrices in $\mathrm{M}_n(D)$.
\end{theorem}
	
\begin{proof}
Assume that $D$ is a field. Let $A\in\mathrm{M}_n(D)$. Then, $\mathrm{sl}_n(D)$ is a linear hyperplane of $\mathrm{M}_n(D)$. Owing to \cite[Theorem 3 and Proposition 12]{Pa_Pa_11}, we conclude that $A$ is a product of two matrices in $\mathrm{sl}_n(D)$, as expected.
\end{proof}

However, for matrices over division rings the situation is rather more complicated. So, for our successful presentation, our next pivotal assertion is the following.

\begin{theorem}[Bruhat decomposition]\label{VUV}
Let $D$ be a division ring and $n\geq 2$ an integer. If $A\in\mathrm{M}_n(D)$, then $A$ can be expressed as $LPHU$ in which $L$ is a lower triangular matrices in $\mathrm{GL}_n(D)$ with the entries on the main diagonal are $1$, $P$ is a permutation matrix in $\mathrm{M}_n(D)$, $H$ is a diagonal matrix in $\mathrm{M}_n(D)$ and $U$ is an upper triangular matrix in $\mathrm{GL}_n(D)$ with the entries on the main diagonal are $1$.
\end{theorem}

\begin{proof}
This can be found in \cite[Theorem 9.2.2, Page 349]{Bo_Co_02}.
\end{proof}

\medskip

Further, to prove our main result in this section, we need a series technical claims as follows:

\begin{lemma}\label{23}
Let $D$ be a division ring. Then, the following statements are true.
\begin{enumerate}[\rm (i)]
\item If $a_1,a_2\in D$, then $\begin{pmatrix}
			a_1&0\\0&a_2
		\end{pmatrix}$ is a product of two traceless matrices in $\mathrm{M}_n(D)$.
		
\item If $a_1,a_2,a_3\in D\setminus\{0\}$, then $\begin{pmatrix}
			a_1&0&0\\0&a_2&0\\0&0&a_3
		\end{pmatrix}$ is a product of three traceless matrices in $\mathrm{M}_n(D)$.
\end{enumerate}
\end{lemma}

\begin{proof}
(i) If $a_1,a_2\in D$, then an easy check shows that $$\begin{pmatrix}
		a_1&0\\0&a_2
	\end{pmatrix}=\begin{pmatrix}
	0&1\\1&0
	\end{pmatrix}\begin{pmatrix}
	0&a_2\\a_1&0
	\end{pmatrix},$$

\medskip

\noindent as required.

If $a_1,a_2,a_3\in D\setminus\{0\}$, then a routine inspection shows that $$\begin{pmatrix}
	a_1&0&0\\0&a_2&0\\0&0&a_3
	\end{pmatrix}=\begin{pmatrix}
	0&a_3&0\\-a_1&a_3&0\\0&0&-a_3
	\end{pmatrix}\begin{pmatrix}
	1&-a_1^{-1}a_2&0\\a_3^{-1}a_1&0&0\\0&0&-1
	\end{pmatrix},$$

\medskip

\noindent as required.
\end{proof}

\begin{lemma}\label{diagonal}
Let $D$ be a division ring and $n\geq 2$ an integer. If the elements $a_1,a_2,\ldots,a_n\in D$, then the diagonal matrix with entries $a_1,a_2,\ldots,a_n$ is a product of two traceless matrices in $\mathrm{M}_n(D)$.
\end{lemma}

\begin{proof}
This claim is directly inferred from Lemma~\ref{23} by simple induction.
\end{proof}

\begin{lemma}\label{UT}
Let $D$ be a division ring and $n\geq 2$ an integer. If $A$ is an upper (resp., lower) triangular matrix in $\mathrm{M}_n(D)$ whose entries on the main diagonal are $1$, then $A$ is a product of at most four traceless matrices in $\mathrm{M}_n(D)$.
\end{lemma}

\begin{proof}
Suppose that $A$ is an upper triangular matrix in $\mathrm{M}_n(D)$ such that the entries on the main diagonal are $1$. By virtue of a plain technical manipulation, which we leave to the interested reader to be inspected, the matrix $A$ can be expressed as the product $BC$ of matrices $B$ and $C$ in which $B$ is a product of two traceless matrices in $\mathrm{M}_n(D)$ and $C$ is the diagonal matrix in $\mathrm{M}_n(D)$ with only one non-zero entry and $1$'s elsewhere on the main diagonal. So, with Lemma~\ref{diagonal} at hand, $A$ is a product of at most four traceless matrices in $\mathrm{M}_n(D)$, as expected.
\end{proof}

We are now ready to prove our main theorem of this section.

\begin{theorem}\label{trace}
Let $D$ be a division ring and $n\geq 2$ an integer. If $A\in\mathrm{M}_n(D)$, then $A$ is a product of at most twelve traceless matrices in $\mathrm{M}_n(D)$.
\end{theorem}

\begin{proof}
Given $A\in\mathrm{M}_n(D)$. Invoking Theorem~\ref{VUV}, $A$ can be expressed as $LPHU$ in which $L$ is a lower triangular matrices in $\mathrm{GL}_n(D)$ having the entries on the main diagonal exactly $1$, $P$ is a permutation matrix in $A\in\mathrm{M}_n(D)$, $H$ is a diagonal matrix and $U$ is an upper triangular matrix in $\mathrm{GL}_n(D)$ possessing the entries on the main diagonal precisely $1$. Now, Lemma~\ref{UT} allows us to derive that $L$ and $U$ are products of two four traceless matrices in $\mathrm{M}_n(D)$. Since $P$ is a matrix over a subfield of $D$, we deduce that $P$ is a product of two traceless matrices in $\mathrm{M}_n(D)$. However, Lemma~\ref{diagonal} leads us to the fact that $H$ is a product of two traceless matrices in $\mathrm{M}_n(D)$. Finally, $A$ is a product of at most twelve traceless matrices in $\mathrm{M}_n(D)$, as stated.
\end{proof}

We end this section with the following two useful statements.

\begin{proposition}
Let $D$ be a division ring. Then, every matrix in $\mathrm{M}_2(D)$ is a product of at most six traceless matrices in $\mathrm{M}_2(D)$.
\end{proposition}
	
\begin{proof}
Choose $A=\begin{pmatrix}
			a&b\\c&d
		\end{pmatrix}\in\mathrm{M}_2(D)$. If $a\neq0$, then $$A=\left(\begin{smallmatrix}
			1&0\\0&-1
		\end{smallmatrix}\right)\left(\begin{smallmatrix}
			1&0\\-ca^{-1}&-1
		\end{smallmatrix}\right)\left(\begin{smallmatrix}
			0&a\\1&0
		\end{smallmatrix}\right)\left(\begin{smallmatrix}
			0&d-ca^{-1}b\\1&0
		\end{smallmatrix}\right)\left(\begin{smallmatrix}
			1&0\\0&-1
		\end{smallmatrix}\right)\left(\begin{smallmatrix}
			1&a^{-1}b\\0&-1
		\end{smallmatrix}\right).$$ If $a=0$ and $b\neq0$, then $$A=\left(\begin{matrix}
			0&b\\c&d
		\end{matrix}\right)=\left(\begin{matrix}
			1&0\\0&-1
		\end{matrix}\right)\left(\begin{matrix}
			1&0\\-db^{-1}&-1
		\end{matrix}\right)\left(\begin{matrix}
			0&b\\c&0
		\end{matrix}\right).$$ If $a=b=0$ and $c\neq0$, then $$A=\left(\begin{matrix}
			0&0\\c&d
		\end{matrix}\right)=\left(\begin{matrix}
			0&0\\c&0
		\end{matrix}\right)\left(\begin{matrix}
			1&0\\0&-1
		\end{matrix}\right)\left(\begin{matrix}
			1&c^{-1}d\\0&-1
		\end{matrix}\right).$$ If $a=b=c=0$, then $$A=\left(\begin{matrix}
			0&0\\0&d
		\end{matrix}\right)=\left(\begin{matrix}
			0&0\\1&0
		\end{matrix}\right)\left(\begin{matrix}
			0&d\\1&0
		\end{matrix}\right).$$ This substantiates our claim.
\end{proof}

\begin{proposition}\label{twoprod}
Let $R$ be a ring and $n\geq 2$ an integer. If $A\in\mathrm{M}_n(R)$, then $A$ is a sum of two products of pairs of traceless matrices in $\mathrm{M}_n(R)$.
\end{proposition}

\begin{proof}
Choose $A=(a_{i,j})\in\mathrm{M}_n(R)$. So, $A$ can be expressed as the sum $A=B+C$ in which $B=(b_{i,j})\in\mathrm{M}_n(R)$ is a lower-triangular matrix and $C=(c_{i,j})\in\mathrm{M}_n(R)$ is an upper-triangular matrix satisfying the equations
		$b_{1,1}=a_{1,1}, c_{1,1}=0,b_{2,2}+c_{2,2}=a_{2,2},b_{3,3}+c_{3,3}=a_{3,3},\ldots,b_{n-1,n-1}+c_{n-1,n-1}=a_{n-1,n-1},b_{n,n}=0,c_{n,n}=a_{n,n}$. Let us now $B_1$ be the matrix $$\begin{pmatrix}
			b&b_{1,1}&0&\cdots&0\\
			0&b_{2,1}&b_{2,2}&\cdots&0\\
			0&b_{3,1}&b_{3,2}&\ddots&0\\
			\vdots&\vdots&\vdots&\ddots&b_{n-1,n-1}\\
			0&b_{n,1}&b_{n,2}&\cdots&b_{n,n-1}
		\end{pmatrix}$$ in which the element $b$ is chosen so that this matrix will have trace $0$, and let $B_2$ be the matrix $$\begin{pmatrix}
		0& & & & \\
		1&0& & & \\
		 &1&\ddots& &\\
		 & &\ddots&0&0\\
		 & & &1&0
		\end{pmatrix}.$$ Also, let $C_1$ be the matrix $$\begin{pmatrix}
		c_{1,2}&c_{1,3}&\cdots&c_{1,n}&0\\
		c_{2,2}&c_{2,3}&\cdots&c_{2,n}&0\\
		0&c_{3,3}&\cdots&c_{3,n}&0\\
		\vdots&\vdots&\ddots&\vdots&\vdots\\
		0&0&\cdots&c_{n,n}&c
		\end{pmatrix},$$ where the element $c$ is chosen so that this matrix will be traceless, and let $C_2$ be the transpose matrix of $B_2$. Therefore, $A=B_1B_2+C_1C_2$ is a sum of two products of pairs of traceless matrices in $\mathrm{M}_n(R)$, as required.
\end{proof}

\section{Products of semi-traceless matrices}
	
Everywhere in this section, let $D$ be a division ring and $n>1$ a natural number and we use the notations: $\mathrm{M}_n(D)$ and $\mathrm{GL}_n(D)$ are respectively the ring of matrices of degree $n$ over $D$ and the general linear group of matrices of degree $n$ over $D$. Assume that $f(x)=x^n+a_{n-1}x^{n-1}+\cdots+a_0\in D[x]$ is a monic polynomial in a variable  $x$ with coefficients $a_i$ in $D$. \textit{A companion matrix} $C(f)$ of $f(x)$ is defined as $$C(f)=\begin{pmatrix}
		0& & &-a_0 \\
		1& & &-a_1\\
		&\ddots& &\vdots\\
		0& &1&-a_{n-1}
	\end{pmatrix}\in \mathrm{M}_n(D).$$

\medskip

We start here with the following.
	
\begin{lemma}\label{2companion}
If $A\in\mathrm{GL}_n(D)$ and $B$ is a companion matrix in $\mathrm{GL}_n(D)$, then $A$ can be expressed as the product $B'C$ of matrices $B'$ and $C$ in which $B'$ is similar to $B$ and $C$ is similar to a  companion matrix in $\mathrm{GL}_n(D)$.
\end{lemma}
	
\begin{proof}
This is a special case of \cite[Theorem 10]{Pa_VaWhe_70}.
\end{proof}

Our basic technicality asserts thus.
	
\begin{lemma}\label{companion traceless}
Each companion matrix in $\mathrm{M}_n(D)$ is a product of two traceless matrices in $\mathrm{M}_n(D)$.
\end{lemma}
	
\begin{proof}
Let $A$ be a companion matrix in $\mathrm{M}_n(D)$, namely $$A=\begin{pmatrix}
			0& & &a_0 \\
			1& & &a_1\\
			&\ddots& &\vdots\\
			0& &1&a_{n-1}
		\end{pmatrix}.$$ First, we consider $n=2$. If $a_0=0$, then $$A=\begin{pmatrix}
			0&0\\1&a_1
		\end{pmatrix}=\begin{pmatrix}
		0&0\\1&0
		\end{pmatrix}\begin{pmatrix}
		1&a_1\\0&-1
		\end{pmatrix}.$$ If $a_0\neq0$, then $$A=\begin{pmatrix}
			0&a_0\\1&a_1
		\end{pmatrix}=\begin{pmatrix}
			a_1&-a_0\\1+a_1a_0^{-1}a_1&-a_1
		\end{pmatrix}\begin{pmatrix}
			1&0\\a_0^{-1}a_1&-1
		\end{pmatrix}$$ is a product two traceless matrices. Now, we consider $n>2$. Then, $A$ can be expressed as the product $BC$ in which  $$B=\begin{pmatrix}
			0&0&0&\cdots&0&-a_0\\
			0&1&0&\cdots&0&-a_1\\
			0&0&1&\cdots&0&-a_2\\
			\vdots&\vdots&\ddots&\ddots&\ddots&\vdots\\
			0&0&\cdots&0&1&-a_{n-2}\\
			1&-1&0&\cdots&0&-(n-2)
		\end{pmatrix}$$ and $$C=\begin{pmatrix}
			1&0&0&\cdots&1&a_{n-1}-(n-2)\\
			1&0&0&\cdots&0&0\\
			0&1&0&\cdots&0&0\\
			\vdots&\ddots&\ddots&\vdots&\vdots\\
			0&0&\cdots&1&0&0\\
			0&0&0&\cdots&0&-1
		\end{pmatrix}$$ are traceless matrices, as required.
\end{proof}

\begin{remark}
Let $R$ be a ring. We know that, if $n\geq3$, then every companion matrix in $\mathrm{M}_n(R)$ is a product of two traceless matrices in $\mathrm{M}_n(R)$. However, if $n=2$, then every companion matrix in $\mathrm{M}_2(R)$ is a product of four traceless matrices in $\mathrm{M}_2(R)$. Indeed, a simple calculation shows that $$\begin{pmatrix}0&a_0\\ \:1&a_1\end{pmatrix}=\begin{pmatrix}0&1\\ 1&0\end{pmatrix}\begin{pmatrix}0&1\\ \:a_0\:&0\end{pmatrix}\begin{pmatrix}0&-1\\ 1&0\end{pmatrix}\begin{pmatrix}1&a_1\\ 0&-1\end{pmatrix}.$$
\end{remark}

\medskip

The next claim is well-known, but is formulated here only for convenience and completeness.
	
\begin{lemma}\label{nilpotent}
Each nilpotent matrix in $\mathrm{M}_n(D)$ is a semi-traceless matrix.
\end{lemma}
	
\begin{proof}
This is implied directly from \cite[Lemma 3.2]{Pa_AbLe_21}.
\end{proof}

The next comments are worthwhile to explain somewhat the more complicated situation here.

\begin{remark}
Let $\mathbb{H}$ be the real quaternion division ring with $i,j,k$ satisfying $i^2=j^2=k^2=-1$ and $ij=-ji=k$. Thus, if $$A=\begin{pmatrix}
			i&j\\-j&i
		\end{pmatrix}\in\mathrm{M}_2(\mathbb{H}),$$ then $A^2=0$ and $A$ is a nilpotent matrix whose trace is $2i\neq 0$. Therefore, there is a nilpotent matrix over a non-commutative division ring which is not traceless.
\end{remark}

Our chief result in the present section is the following one.
	
\begin{theorem}\label{semi-trace}
If $D$ is a non-commutative division ring, then each matrix in $\mathrm{M}_n(D)$ is a product of at most four semi-traceless matrices in $\mathrm{M}_n(D)$.
\end{theorem}
	
\begin{proof}
Assume that $D$ is a non-commutative division ring. Let $A\in\mathrm{M}_n(D)$. If $A\in\mathrm{GL}_n(D)$, then, combining Lemma~\ref{2companion} and Lemma~\ref{companion traceless}, one verifies that $A$ is a product of three semi-traceless matrices in $\mathrm{M}_n(D)$. Now, we consider $A\notin\mathrm{GL}_n(D)$. According to \cite[Corollary 7]{Pa_Ca_16}, there exist $B\in\mathrm{GL}_n(D)$ and $C$ is nilpotent in $\mathrm{M}_n(D)$ such that $A$ can be expressed as the product of $B$ and $C$, writing $A=BC$. In conjunction with the previous argument, one checks that $B$ is a product of three semi-traceless matrices in $\mathrm{M}_n(D)$. On the other hand, Lemma~\ref{nilpotent} applies to get that $C$ is a semi-traceless matrix. Therefore, $A$ is a product of four semi-traceless matrices in $\mathrm{M}_n(D)$, as promised.
\end{proof}

A question which immediately arises is whether the estimation of the number of matrices in the decomposition is the exact one or, in other words, can we decrease the number of matrices in this decomposition? In this aspect, we can just offer the following.

\begin{proposition}\label{semi-2}
Let $D$ be division ring with center $F$ and $n\geq 2$ an integer. If $F$ is perfect, then every matrix in $\mathrm{M}_n(D)$ is a product of two semi-traceless matrices in $\mathrm{M}_n(D)$.
\end{proposition}

\begin{proof}
If $F$ is perfect, then employing \cite[Proposition 1]{Bo_Si_75}, we can detect that every matrix in $\mathrm{M}_n(D)$ is similar to a matrix over the subfield of $D$ containing $F$. Furthermore, consulting with Theorem~\ref{field}, each matrix in $\mathrm{M}_n(D)$ is a product of two semi-traceless matrices in $\mathrm{M}_n(D)$, as required.
\end{proof}

\section{Finitary matrices}

Let $D$ be a division ring. Denote by $\mathrm{M}_\infty(D)$ the ring of all \textit{finitary matrices}, that is, countably infinite matrices with only finitely many non-zero entries. Note that for any matrix $A\in\mathrm{M}_\infty(D)$, we may find a positive integer $k$ such that $A$ can be expressed as the block matrix $$A=\begin{pmatrix}
	A'&0\\0&0
\end{pmatrix},$$ where $A'\in\mathrm{M}_k(D)$.  So, $A$ is called a traceless matrix in $\mathrm{M}_\infty(D)$ if $A'$ is a traceless matrix in $\mathrm{M}_k(D)$.

Let $F$ be a field and let $F\langle\mathfrak{X}\rangle$ be the free algebra generated by the set $\mathfrak{X}=\{\mathbf{x}_1,\mathbf{x}_2,\ldots\}$, that is, the algebra of non-commutative polynomials in the variables $\mathbf{x}_1, \mathbf{x}_2,\ldots$. For any $F$-algebra $A$ and $f=f(\mathbf{x}_1,\ldots,\mathbf{x}_m)\in F\langle\mathfrak{X}\rangle$, let $f(A)=\{f(a_1,\ldots,a_m)\mid a_1,\ldots,a_m\in A\}$ and we call $f(A)$ the \textit{image} of $f$. Note that $f(A)$ is invariant under conjugation. The polynomial $f=f(\mathbf{x}_1,\ldots,\mathbf{x}_m)\in F\langle\mathfrak{X}\rangle$ is called \textit{multilinear} if $f$ is of the form:
$$f=\sum_{\sigma\in S_m}\lambda_\sigma\mathbf{x}_{\sigma(1)}\mathbf{x}_{\sigma(2)}\cdots\mathbf{x}_{\sigma(m)}$$ in which $\lambda_\sigma\in F$ and $S_m$ is the symmetric group of degree $m$.

\begin{proposition} {\rm\cite[Corollary 1.2]{Pa_Vi_21}}\label{Vitas}
Let $F$ be an infinite field and let $f\in F\langle\mathfrak{X}\rangle$ be a non-zero multilinear polynomial. Then, any traceless finitary matrices over $F$ is a image of $f$ evaluated on $\mathrm{M}_\infty(F)$.
\end{proposition}

We thus have to following corollary to Theorem~\ref{field} and Proposition~\ref{Vitas}.

\begin{corollary}\label{Co_Vitas}
Let $F$ be an infinite field. Then, any matrix in $\mathrm{M}_\infty(F)$ is a product of two traceless matrices in $\mathrm{M}_\infty(F)$. As a result,  any matrix in $\mathrm{M}_\infty(F)$ is a product of two images of non-zero multilinear polynomials in $F\langle\mathfrak{X}\rangle$ evaluated on $\mathrm{M}_\infty(F)$.
\end{corollary}

From Theorem~\ref{trace}, we conclude that if we let $D$ to be a non-commutative division ring, then any matrix in $\mathrm{M}_\infty(F)$ is a product of twelve traceless matrices in $\mathrm{M}_\infty(D)$. Regarding semi-traceless matrices, we can decrease the number of matrices in the decomposition of Theorem~\ref{semi-trace}.

\medskip

We now manage to establish the following two chief results.

\begin{theorem}\label{new1}
Let $D$ be a non-commutative division ring. Then, any matrix in $\mathrm{M}_\infty(F)$ is a product of two semi-traceless matrices in $\mathrm{M}_\infty(D)$.
\end{theorem}

\begin{proof}
Let $A\in\mathrm{M}_\infty(D)$. Then, there exists a positive integer $k$ such that $A$ can be expressed as the block matrix $$A=\begin{pmatrix}
		A'&0\\0&0
	\end{pmatrix},$$ where $A'\in\mathrm{M}_k(D)$. According to \cite[Section 8.4, Page 505]{Bo_Co_85}, $A'$ is similar to the block diagonal matrix $A'_{k_1}\oplus\cdots\oplus A'_{k_t}$ of companion matrices $A'_{k_1}\in\mathrm{M}_{k_1}(D),\ldots,A'_{k_t}\in\mathrm{M}_{k_t}(D)$ in which $k_1+\cdots+k_t=k$, that is, there exists $P\in\mathrm{GL}_k(D)$ such that $P^{-1}A'P=A'_{k_1}\oplus\cdots\oplus A'_{k_t}$. We divide the proof into the following cases:

\medskip
	
\noindent\textbf{Case 1:} If $k_1,\ldots,k_t$ are all different from $1$, then by Lemma~\ref{companion traceless}, $A'$ is a product of two semi-traceless matrices in $\mathrm{M}_k(D)$. Therefore, $A$ is a product of two semi-traceless matrices in $\mathrm{M}_\infty(D)$.

\medskip
	
\noindent\textbf{Case 2:} There is only $k_i\in\{k_1,\ldots,k_t\}$ such that $k_i=1$. Without loss of generality, we can assume $k_t=1$. Instead of considering $A'$, we set $A''=\begin{pmatrix}
		A'&0\\0&0
	\end{pmatrix}\in\mathrm{M}_{k+1}(D)$. Then, one verifies that $$\begin{pmatrix}
	P&0\\0&1
	\end{pmatrix}^{-1}A''\begin{pmatrix}
	P&0\\0&1
	\end{pmatrix}=A'_{k_1}\oplus\cdots\oplus A'_{k_{t-1}}\oplus\begin{pmatrix}
	A'_{k_t}&0\\0&0
	\end{pmatrix}.$$ Note that $\begin{pmatrix}
	A'_{k_t}&0\\0&0
	\end{pmatrix}$ is a diagonal matrix in $\mathrm{M}_2(D)$ and $A'_{k_1},\ldots,A'_{k_{t-1}}$ are companion matrices of size greater than $1$. According to Lemma~\ref{23} and Lemma~\ref{companion traceless}, $A''$ is a product of two traceless matrices in $\mathrm{M}_{k+1}(D)$. Therefore, $A$ is a product of two semi-traceless matrices in $\mathrm{M}_\infty(D)$.

\medskip
	
\noindent\textbf{Case 3:} There is at least two $k_i\in\{k_1,\ldots,k_t\}$ such that $k_i= 1$. Similarly,  by using Lemma~\ref{23} and Lemma~\ref{companion traceless}, $A'$ is a product of two traceless matrices in $\mathrm{M}_{k}(D)$. Therefore, $A$ is a product of two semi-traceless matrices in $\mathrm{M}_\infty(D)$.
\end{proof}

\begin{theorem}\label{image1}
Let $D$ be a non-commutative division ring with center $F$. Then, any matrix in $\mathrm{M}_\infty(D)$ is a product of at most seven images of non-zero multilinear polynomials in $F\langle\mathfrak{X}\rangle$ evaluated on $\mathrm{M}_\infty(D)$.
\end{theorem}

To prove Theorem~\ref{image1}, we need a series technical claims as follows:

\medskip

In \cite[Lemma 2.1]{Pa_EgGo_19}, it was proven that if $D$ is finite dimensional over its center, then every non-central matrix in $\mathrm{GL}_n(D)$ is similar to $XHY$ where $X$ is a lower triangular matrix with the main diagonal entries are $1$ and $Y$ is an upper triangular matrix with the main diagonal entries are $1$ and $H$ is the diagonal matrix with the main diagonal entries are $1,1,\dots,1,h$ for some $h\in D\setminus\{0\}$. In fact, this result still holds for arbitrary division rings, because the technique of the proof does not use the finiteness of their dimensions.

\medskip

We, thereby, arrive at the following.

\begin{lemma}\label{lemma_cheohoa} \cite[Lemma 2.1]{Pa_EgGo_19}  Let $D$ be a division ring and $n\ge 2$ an integer. If  $A\in \mathrm{GL}_n(D)$ is non-central, then there exists $P\in \mathrm{GL}_n(D)$ such that $P^{-1}AP$ has the form $$P^{-1}AP=XHY,$$ where $X$ is a lower triangular matrix with the main diagonal entries are $1$ and $Y$ is an upper triangular matrix with the main diagonal entries are $1$ and $H$ is the diagonal matrix with the main diagonal entries are $1,1,\dots,1,h$ for some $h\in D\setminus\{0\}$.
\end{lemma}

Let $D$ be a division ring. For two elements $\alpha,\beta \in D$, denote by $$J_{n}(\alpha,\beta)=\begin{pmatrix}
	\alpha&\beta&0&\cdots&0&0&0\\
	0&\alpha&\beta&\cdots&0&0&0\\
	0&0&\alpha&\cdots&0&0&0\\
	\vdots&\vdots&\vdots&\ddots&\vdots&\vdots&\vdots\\
	0&0&0&\cdots&\alpha&\beta&0\\
	0&0&0&\cdots&0&\alpha&\beta\\
	0&0&0&\cdots&0&0&\alpha
\end{pmatrix}
\in \mathrm{M}_{n}(D)$$ is the upper triangular matrix with entries the main diagonal and the first super diagonal are respectively $\alpha$ and $\beta$. For short, we will just write $J_{n}(\alpha)$ for $J_{n}(\alpha,1)$.

\medskip

We also need the following technicality which we list below for completeness of the exposition.

\begin{lemma}\cite[Theorem 7]{Pa_Djo_85} \label{Djo}
Let $D$ be a division ring with center $F$, $n$ a positive integer and $A\in \mathrm{M}_n(D)$. If $A$ is triangulate and algebraic over $F$, then there exist positive integers $s, m_1,\dots,m_s$ and elements $\alpha_1,\alpha_2,\dots,\alpha_s,\beta_1,\beta_2,\dots,\beta_s\in D$ such that $A$ is similar to $\bigoplus_{i=1}^sJ_{m_i}(\alpha_i,\beta_i)$ in which $m_1+m_2+\dots+m_s=n$ and $\beta_i\notin \{\alpha_i a-a\alpha_i\mid a\in D\}$ for every $i=1,2,\dots,s$. Additionally, for each $i=1,2,\dots, s$, if $\alpha_i$ is separable over $F$, then $\beta_i$ may be chosen to be $1$.
\end{lemma}

We now have all the ingredients necessary to show the truthfulness of the following key result.

\begin{proposition}\label{UTLT}
Let $D$ be a division ring and $n\geq 2$ an integer. If $A\in\mathrm{M}_n(D)$ is unipotent, then $A$ is similar to the diagonal block matrix $\bigoplus_{i=1}^s J_{m_i}(1)$ in which $m_1+m_2+\dots+m_s=n$.
\end{proposition}

\begin{proof}
Let $A\in\mathrm{M}_n(D)$ be unipotent. Then, $\mathrm I_n-A$ is nilpotent. Thanks to \cite[Lemma 3.2]{Pa_AbLe_21}, one sees that $\mathrm I_n-A$ is similar to a strictly upper triangular matrix, which implies that $A$ is similar to a upper triangular matrix whose diagonal entries are all $1$. Employing Lemma~\ref{Djo}, one checks that $A$ is similar to the diagonal block matrix $\bigoplus_{i=1}^s J_{m_i}(\alpha_i,\beta_i)$ in which $m_1+m_2+\dots+m_s=n$ and $\beta_i\notin \{\alpha_i a-a\alpha_i\mid a\in D\}$ for every $i=1,2,\dots,s$. Moreover, since $A$ is similar to a upper triangular matrix whose diagonal entries are all $1$, we have that $\alpha_i$'s are $1$. Furthermore, using Lemma~\ref{Djo} again, we deduce that $\beta_i$'s are $1$. Therefore, $A$ is similar to the diagonal block matrix $\bigoplus_{i=1}^s J_{m_i}(1)$, as required.
\end{proof}

We now continue with

\begin{proof}[The proof of Theorem~\ref{image1}]
Let $A\in\mathrm{M}_\infty(D)$. Then, there exists a positive integer $n$ such that $A$ can be expressed as the block matrix $$A=\begin{pmatrix}
		A'&0\\0&0
	\end{pmatrix},$$ where $A'\in\mathrm{M}_n(D)$. By using \cite[Corollary 7]{Pa_Ca_16}, there exist $G\in\mathrm{GL}_n(D)$ and $N$ is nilpotent in $\mathrm{M}_n(D)$ such that $A$ can be expressed as the product of $G$ and $N$, writing $A=GN$. If $G$ is central, then by\cite[$\S21$, Theorem 1, Page 140]{Bo_Dra_83}, there exists $\lambda\in F$ such that $G=\lambda\mathrm{I}_n$. Therefore, if $G$ is central, then by using Corollary~\ref{Co_Vitas}, the block matrix $$\begin{pmatrix}
	G&0\\0&0
	\end{pmatrix}$$ is a product of two images of non-zero multilinear polynomials in $F\langle\mathfrak{X}\rangle$ evaluated on $\mathrm{M}_\infty(D)$. Now, we consider $G$ is non-central. According to Proposition~\ref{lemma_cheohoa}, there exists $P\in \mathrm{GL}_n(D)$ such that $P^{-1}GP$ has the form $$P^{-1}GP=XHY,$$ where $X$ is a lower triangular matrix with the main diagonal entries are $1$ and $Y$ is an upper triangular matrix with the main diagonal entries are $1$ and $H$ is the diagonal matrix with the main diagonal entries are $1,1,\dots,1,h$ for some $h\in D\setminus\{0\}$. According to Proposition~\ref{UTLT}, both $X$ and $Y$ are similar to matrices over $F$. On the other hand, $H$ is a matrix over the subfield $F(h)$ generated by $h$ over $F$. Again, by using Corollary~\ref{Co_Vitas}, the block matrices $\begin{pmatrix}
	X&0\\0&0
	\end{pmatrix},\begin{pmatrix}
	Y&0\\0&0
	\end{pmatrix}$ and $\begin{pmatrix}
	H&0\\0&0
	\end{pmatrix}$ are products of two images of non-zero multilinear polynomials in $F\langle\mathfrak{X}\rangle$ evaluated on $\mathrm{M}_\infty(D)$. Therefore, the block matrix $\begin{pmatrix}
	G&0\\0&0
	\end{pmatrix}$ is a product of six images of non-zero multilinear polynomials in $F\langle\mathfrak{X}\rangle$ evaluated on $\mathrm{M}_\infty(D)$. Furthermore, by using \cite[Lemma 3.2]{Pa_AbLe_21}, $N$ is similar to a traceless matrix over $F$, and applying Proposition~\ref{Vitas}, the block matrix $\begin{pmatrix}
	N&0\\0&0
	\end{pmatrix}$ is a image of non-zero multilinear polynomials in $F\langle\mathfrak{X}\rangle$ evaluated on $\mathrm{M}_\infty(D)$. Consequently, $A$ is a product of seven images of non-zero multilinear polynomials in $F\langle\mathfrak{X}\rangle$ evaluated on $\mathrm{M}_\infty(D)$, as needed.
\end{proof}

In view of Proposition~\ref{semi-2}, and by using Proposition~\ref{Vitas}, we deduce the following result.

\begin{proposition}\label{Minfty}
Let $D$ be division ring with center $F$. If $F$ is perfect, then every matrix in $\mathrm{M}_\infty(D)$ is a product of two images of non-zero multilinear polynomials in $F\langle\mathfrak{X}\rangle$ evaluated on $\mathrm{M}_\infty(D)$.
\end{proposition}

We now arrange to prove the following major result.

\begin{theorem}\label{image2}
Let $D$ be a non-commutative division ring which is finite-dimensional over the center $F$ of $D$. Then, any matrix in $\mathrm{M}_\infty(D)$ is a product of at most five images of non-zero multilinear polynomials in $F\langle\mathfrak{X}\rangle$ evaluated on $\mathrm{M}_\infty(D)$.
\end{theorem}

\begin{proof}
Let $A\in\mathrm{M}_\infty(D)$. Then, there exists a positive integer $n$ such that $A$ can be expressed as the block matrix $$A=\begin{pmatrix}
		A'&0\\0&0
	\end{pmatrix},$$ where $A'\in\mathrm{M}_n(D)$. The first of this proof is similar to the proof of Theorem~\ref{image1}. Invoking \cite[Corollary 7]{Pa_Ca_16}, there exist $G\in\mathrm{GL}_n(D)$ and $N$ is nilpotent in $\mathrm{M}_n(D)$ such that $A$ can be expressed as the product of $G$ and $N$, writing $A=GN$. Then, $\begin{pmatrix}
	N&0\\0&0
	\end{pmatrix}$ is a image of non-zero multilinear polynomials in $F\langle\mathfrak{X}\rangle$ evaluated on $\mathrm{M}_\infty(D)$. If $G$ is central, then $$\begin{pmatrix}
		G&0\\0&0
	\end{pmatrix}$$ is a product of two images of non-zero multilinear polynomials in $F\langle\mathfrak{X}\rangle$ evaluated on $\mathrm{M}_\infty(D)$. Now, we consider $G$ is non-central.	According to Lemma~\ref{lemma_cheohoa}, there exists $P\in \mathrm{GL}_n(D)$ such that $P^{-1}GP$ has the form $$P^{-1}GP=XY$$ where $X$ is a lower triangular matrix with the main diagonal entries are $1$ and $Y$ is an upper triangular matrix with the main diagonal entries are  $1,1,\dots,1,h$ for some $h\in D\setminus\{0\}$. If $h=1$, then by using the similar argument in the proof of Theorem~\ref{image1}, both $X$ and $Y$ are products of two images of non-zero multilinear polynomials in $F\langle\mathfrak{X}\rangle$ evaluated on $\mathrm{M}_\infty(D)$. Now, we consider $h\neq1$. By increasing $n$ if necessary, we can assume $n$ is even. Since $D$ is a non-commutative division ring which is finite-dimensional over $F$, the center $F$ of $D$ is infinite. If $h$ and $h^{-1}$ are conjugate, then by using \cite[Lemma 2.3]{Pa_BiDuHa_22},  one can choose $\alpha\in F\setminus\{0\}$ such that $(\alpha h)$ and $(\alpha h)^{-1}$ are non-conjugate, so without loss of generality, we can assume $h$ and $h^{-1}$ are non-conjugate. On the other hand, by putting $t=\frac{n-2}{2}$ and using \cite[Lemma 2.3]{Pa_BiDuHa_22} again, since $F$ is infinite, we can choose $x_1,x_1^{-1},\ldots,x_t,x_t^{-1}\in F\setminus\{-1,0,1\}$ such that $$x_1,x_1^{-1},\ldots,x_t,x_t^{-1},h,h^{-1},1$$  are pairwise non-conjugate. Let $Q$ be the diagonal matrix with entries on the main diagonal are $$x_1,x_1^{-1},\ldots,x_t,x_t^{-1},h,1.$$ By using \cite[Lemma 3.2]{Pa_BiDuHaSo_2022}, the matrix $XQ$ is similar to $Q$ and the matrix $Q^{-1}Y$ is similar to the diagonal matrix $Q'$ with entries on the main diagonal are $$x_1^{-1},x_1,\ldots,x_t^{-1},x_t,h^{-1},h.$$ Note that both $Q$ and $Q'$ are matrices over the field $F(h)$ generated by $h$ over $F$. Therefore, both block matrices $\begin{pmatrix}
	X&0\\0&0
	\end{pmatrix}$ and $\begin{pmatrix}
	Y&0\\0&0
	\end{pmatrix}$ are products of two images of non-zero multilinear polynomials in $F\langle\mathfrak{X}\rangle$ evaluated on $\mathrm{M}_\infty(D)$. Consequently, $A$ is a product of at most five images of non-zero multilinear polynomials in $F\langle\mathfrak{X}\rangle$ evaluated on $\mathrm{M}_\infty(D)$, as stated.
\end{proof}

\section{Images of non-commutative polynomials}	

Continuing the previous section, in this section, we will firstly study products of images of multilinear polynomials and then products of images of non-commutative polynomials with zero constant terms. Throughout this section, we use all denotations in the preceding section.

\medskip

First, we will begin with

\begin{proposition}\label{H}
Let $\mathbb{H}$ be the real quaternion division ring with $i,j,k$ satisfying $i^2=j^2=k^2=-1$ and $ij=-ji=k$. Then, every element in $\mathbb{H}$ is a product of two images of non-zero and non-central multilinear polynomials in $\mathbb{R}\langle\mathfrak{X}\rangle$ evaluated on $\mathbb{H}$.
\end{proposition}
	
\begin{proof}
Set $\alpha\in\mathbb{H}$ and write $\alpha=x+yi+zj+tk$ for some $x,y,z,t\in\mathbb{R}$. By using \cite[Lemma 2.1]{Pa_Zha_97}, there exists $p\in\mathbb{H}\setminus\{0\}$ such that $$p^{-1}\alpha p=x+\sqrt{y^2+z^2+t^2}i.$$ On the other hand, $$x+\sqrt{y^2+z^2+t^2}i=j(-xj+\sqrt{y^2+z^2+t^2}k).$$ According to \cite[Theorem 6, Page 12]{Bo_Al_16}, the elements $j$ and $-xj+\sqrt{y^2+z^2+t^2}k$ are images of non-zero and non-central multilinear polynomials in $\mathbb{R}\langle\mathfrak{X}\rangle$ evaluated on $\mathbb{H}$. Therefore, $\alpha$ is a product of two images of non-zero and non-central multilinear polynomials in $\mathbb{R}\langle\mathfrak{X}\rangle$ evaluated on $\mathbb{H}$, as required.
\end{proof}	

Similarly to Proposition~\ref{Minfty}, we obtain an analogous version of the assertion for finite matrices over the real quaternion division ring, which states the following.

\begin{corollary}
Let $\mathbb{H}$ be the real quaternion division ring with $i,j,k$ satisfying $i^2=j^2=k^2=-1$ and $ij=-ji=k$ and $n\geq 1$ a positive integer. Then, every element in $\mathrm{M}_n(\mathbb{H})$ is a product of four images of non-zero and non-central multilinear polynomials in $\mathbb{R}\langle\mathfrak{X}\rangle$ evaluated on $\mathrm{M}_n(\mathbb{H})$.
\end{corollary}
	
\begin{proof}
Put $A\in\mathrm{M}_n(\mathbb{H})$. Owing to \cite[Theorem 5.5.3, Page 98]{Bo_Ro_14}, the matrix $A$ is similar to a matrix over the field of complex numbers. According to \cite[Theorem 2.1]{Pa_Bo_98}, $A$ is a product of two diagonalizable matrices, writing $A=BC$. On the other side, invoking Proposition~\ref{H}, any diagonal matrix is a product of two images of non-zero and non-central multilinear polynomials in $\mathbb{R}\langle\mathfrak{X}\rangle$ evaluated on $\mathrm{M}_n(\mathbb{H})$, so are $B$ and $C$. Consequently, $A$ is a product of four images of non-zero and non-central multilinear polynomials in $\mathbb{R}\langle\mathfrak{X}\rangle$ evaluated on $\mathrm{M}_n(\mathbb{H})$, as needed.
\end{proof}

The next two propositions are pivotal for our further results.

\begin{proposition}\label{field2}
Let $F$ be a field. Then, every element in $\mathrm{M}_2(F)$ is a product of two  images of non-zero and non-central multilinear polynomials in $F\langle\mathfrak{X}\rangle$ evaluated on $\mathrm{M}_2(F)$.
\end{proposition}

\begin{proof}
Let $A\in\mathrm{M}_2(F)$. Thus, Theorem~\ref{field} tells us that $A$ is a product of two traceless matrices in $\mathrm{M}_2(F)$. Consulting with \cite[Theorem 1]{Pa_Ma_14}, any traceless matrix is a image of non-zero and non-central multilinear polynomials in $F\langle\mathfrak{X}\rangle$ evaluated on $\mathrm{M}_2(F)$. Therefore, $A$ is a product of two  images of non-zero and non-central multilinear polynomials in $F\langle\mathfrak{X}\rangle$ evaluated on $\mathrm{M}_2(F)$, as required.
\end{proof}

\begin{proposition}
Let $D$ be a division ring with center $F$. Then, every element in $\mathrm{M}_2(D)$ is a product of seven images of non-zero and non-central multilinear polynomials in $F\langle\mathfrak{X}\rangle$ evaluated on $\mathrm{M}_2(D)$. In particular, if $D$ is finite dimensional over $F$, then every element in $\mathrm{M}_2(D)$ is a product of five  images of non-zero and non-central multilinear polynomials in $F\langle\mathfrak{X}\rangle$ evaluated on $\mathrm{M}_2(D)$.
\end{proposition}

\begin{proof}
The first part of the proof uses a similar argument as in the proof of Theorem~\ref{image1}. In fact, let $A\in\mathrm{M}_2(D)$. Using \cite[Corollary 7]{Pa_Ca_16}, there exist $G\in\mathrm{GL}_2(D)$ and a nilpotent $N$ in $\mathrm{M}_2(D)$ such that $A$ can be expressed as the product of $G$ and $N$, writing $A=GN$. If $G$ is central, then in view of \cite[$\S21$, Theorem 1, Page 140]{Bo_Dra_83}, there exists $\lambda\in F$ such that $G=\lambda\mathrm{I}_2$. An application of Proposition~\ref{field2} insures that $G$ is a product of two images of non-zero and non-central multilinear polynomials in $F\langle\mathfrak{X}\rangle$ evaluated on $\mathrm{M}_2(D)$. Moreover, in accordance with Lemma~\ref{lemma_cheohoa}, there exists $P\in \mathrm{GL}_n(D)$ such that $P^{-1}GP$ has the form $$P^{-1}GP=XHY,$$ where $X$ is a lower triangular matrix such that the main diagonal entries are $1$, and $Y$ is an upper triangular matrix with the main diagonal entries $1$, and $H$ is the diagonal matrix with the main diagonal entries $1,h$ for some $h\in D\setminus\{0\}$. Exploiting Lemma~\ref{UTLT}, both $X$ and $Y$ are similar to matrices over $F$. On the other side, $H$ is a matrix over the subfield $F(h)$ generated by $h$ over $F$. Furthermore, employing \cite[Lemma 3.2]{Pa_AbLe_21}, $N$ is similar to a traceless matrix over $F$, so $N$ is a image of a non-zero and non-central multilinear polynomial in $F\langle\mathfrak{X}\rangle$ evaluated on $\mathrm{M}_2(D)$ in view of
\cite[Theorem 1]{Pa_Ma_14}. According to Proposition~\ref{field2}, the matrices $X,Y$ and $H$ are products of two images of non-zero and non-central multilinear polynomials in $F\langle\mathfrak{X}\rangle$ evaluated on $\mathrm{M}_2(D)$. Therefore, $A$ is a product of seven images of non-zero and non-central multilinear polynomials in $F\langle\mathfrak{X}\rangle$ evaluated on $\mathrm{M}_2(D)$, as required.

Concerning the second part, in case that $D$ is finite dimensional over $F$, the proof is completed by the usage of an analogous argument as that from the proof of Theorem~\ref{image2}.
\end{proof}

It is well-known in \cite{Pa_DuSo} that every matrix over an arbitrary field is a product of two additive commutators by noticing that the arguments are similar to Theorem~\ref{field} and knowing the fact that every traceless matrix over a field is an additive commutator. In virtue of Theorem~\ref{field}, Theorem~\ref{semi-trace} and \cite[Theorem 2.4]{Pa_AmRo_94}, we extract the following surprising result.

\begin{proposition}\label{commutator}
Each matrix over a division ring is a product of at most four additive commutators.
\end{proposition}

Next, we consider generalized commutators in matrix rings. Recall that any element of form $abc-cba$ in a ring with unity, containing the elements $a,b,c$, is called a \textit{generalized commutator} (see, e.g., \cite{KL} or \cite{Pa_DaLe_22}, respectively). It is clear that every additive commutator is a generalized commutator. Therefore, we have immediately from Proposition~\ref{commutator} the following interesting and non-trivial consequence.

\begin{corollary}\label{gener}
Each matrix over a division ring is a product of at most four generalized commutators.
\end{corollary}

We are now planning to show the validity of the following main assertion.

\begin{theorem}\label{main3}
Let $R$ be a finite-dimensional algebra over a field $F$ of characteristic $0$ which has no direct summands that are fields in the Wedderburn decomposition. Then, every element in $R$ is a product of at most four generalized commutators in $R$.
\end{theorem}

\begin{proof}
Choose $\alpha\in R$. Since $F$ has characteristic $0$, the Wedderburn decomposition of $R$ is the following $$R\cong\mathrm{M}_{n_1}(D_1)\times\cdots\times\mathrm{M}_{n_t}(D_t)$$ in which the numbers $n_i$'s are positive integers and the $D_i$'s are division rings that are finite-dimensional over $F$. By assumptions, if there exists $i\in\{1,\ldots,t\}$ such that $n_i=1$, then $D_i$ is not a field and so \cite[Corollary 3.8]{Pa_Gup_09} ensures that every element in $D_i$ is a generalized commutator in $D_i$. If, however, there exists an index $i\in\{1,\ldots,t\}$ such that $n_i\neq1$, then, in conjunction with Corollary~\ref{gener}, we observe that every matrix in $\mathrm{M}_{n_i}(D_i)$ is a product of at most four generalized commutators in $\mathrm{M}_{n_i}(D_i)$. Therefore, we can conclude that each element of $R$ is a product of at most four generalized commutators of $R$, as stated.
\end{proof}

Note that, under the circumstances of Theorem~\ref{main3}, for a finite-dimensional algebras whose elements are products of unipotents we refer the interested reader to \cite{BDRS}.

\medskip

The last two statements of ours manifestly demonstrate what happens in the case of an algebraically closed field.

\begin{proposition}\label{algebraically}
Let $F$ be an algebraically closed field and $n>1$ an integer. Then, every element in $\mathrm{M}_n(F)$ is a product of two images of non-commutative polynomials with zero constant terms in $F\langle\mathfrak{X}\rangle$ evaluated on $\mathrm{M}_n(F)$ in which the evaluations of such polynomials on $F$ contain at least a non-zero element.
\end{proposition}

\begin{proof}
Let $A\in\mathrm{M}_n(F)$. Since $F$ is an algebraically closed field, $F$ is infinite. According to \cite[Theorem 2.1]{Pa_Bo_98} and \cite[Theorem 2.2]{Pa_Bo_99}, $A$ is a product of two diagonalizable matrices. On the the hand, by \cite[Lemma 3.1]{Pa_WaZhLu_21}, any diagonal matrix is a image of a non-commutative polynomial with zero constant term in $F\langle\mathfrak{X}\rangle$ evaluated on $\mathrm{M}_n(F)$ in which the evaluation of such a polynomial on $F$ contains at least a non-zero element, and hence so is any diagonalizable matrix. Thus, the matrix $A$ is a product of two images of non-commutative polynomials with zero constant terms in $F\langle\mathfrak{X}\rangle$ evaluated on $\mathrm{M}_n(F)$ in which the evaluations of such polynomials on $F$ contain at least a non-zero element, as needed.
\end{proof}

What we now derive as a consequence is the following curious claim.

\begin{corollary}\label{algebraic}
Let $F$ be an algebraically closed field of characteristic $0$ and let $A$ be a locally finite $F$-algebra. Then, every element in $A$ is a product of two images of non-commutative polynomials with zero constant terms in $F\langle\mathfrak{X}\rangle$ evaluated on $A$ in which the evaluations of such polynomials on $F$ contain at least a non-zero element.
\end{corollary}

\begin{proof}
Choose $\alpha\in A$. With no loss of generality, we may assume that $A$ is a finite dimensional $F$-algebra. Since $F$ is an algebraically closed field of characteristic $0$, by Wedderburn-Artin theorem, $$R\cong\mathrm{M}_{n_1}(F)\times\cdots\times\mathrm{M}_{n_t}(F),$$ where the $n_i$'s are positive integers. Now, given $$\alpha=(\alpha_1,\cdots,\alpha_t)\in\mathrm{M}_{n_1}(F)\times\cdots\times\mathrm{M}_{n_t}(F)$$ in which $\alpha_i\in\mathrm{M}_{n_i}(F)$ for each $1\leq i\leq t$. In accordance with \cite[Lemma 3.1]{Pa_WaZhLu_21} and Proposition~\ref{algebraically}, each element $\alpha_i$ is a product of two images of non-commutative polynomials with zero constant terms in $F\langle\mathfrak{X}\rangle$ evaluated on $\mathrm{M}_n(F)$ in which the evaluations of such polynomials on $F$ contain at least a non-zero element. So, $\alpha$ is also a product of two images of non-commutative polynomials with zero constant terms in $F\langle\mathfrak{X}\rangle$ evaluated on $\mathrm{M}_n(F)$ in which the evaluations of such polynomials on $F$ contain at least a non-zero element, as required.
\end{proof}

We end up the investigation with the following challenging question of some interest and importance.

\begin{problem}\label{unsolved} Decide what are the traceless and semi-traceless matrices in the Ver\v{s}hik-Kerov group, and find a decomposition of such a group into products of images of non-commutative polynomials.
\end{problem}

\noindent{\bf Acknowledgement.} The authors are very thankful to Professor Zachary Mesyan from the University of Colorado at Colorado Springs for the valuable private communications on the present subject.

\medskip
	
\noindent{\bf Funding:} The first-named author, P.V. Danchev, of this research paper was partially supported by the Junta de Andaluc\'ia under Grant FQM 264, and by the BIDEB 2221 of T\"UB\'ITAK.

\end{document}